\documentclass[11pt, twoside
]{article}

\usepackage{amsmath, amsthm, amssymb,bm}
\usepackage{color}
\usepackage[enableskew]{youngtab}

\numberwithin{equation}{section}

\newcommand{\bQ}{\mathbb{Q}}

\newcommand{\bZ}{\mathbb{Z}}
\newcommand{\bE}{\mathbb{E}}

\newcommand{\mf}[1]{\mathfrak{#1}}

\newcommand{\mcal}[1]{\mathcal{#1}}

\def\({ \left( }
\def\){ \right)}

\DeclareMathOperator{\sgn}{sgn}

\theoremstyle{plain}
\newtheorem{thm}{Theorem}[section]
\newtheorem{prop}[thm]{Proposition}
\newtheorem{lemma}[thm]{Lemma}
\newtheorem{corollary}[thm]{Corollary}
\newtheorem{definition}[thm]{Definition}
\newtheorem{example}[thm]{Example}
\newtheorem{remark}[thm]{Remark}
\theoremstyle{conjecture}
\newtheorem{conj}[thm]{Conjecture}
\theoremstyle{problem}

\theoremstyle{formula}

\title{ {\bf Polynomiality of shifted Plancherel averages
and content evaluations }
}
\author{\textsc{Sho Matsumoto}%
\footnote{The author is supported by JSPS KAKENHI 
Grant Number 25800062, 2434003.}
}

\date{\empty}

\begin{document}
%>>>>>>>>>>>>>>>>>>>>>>>>>>>>>>>>>>>>>>>>>

\maketitle

\begin{abstract}
The shifted Plancherel measure
is a natural probability measure on strict partitions. 
We prove a polynomiality  property for 
the averages of the shifted Plancherel measure. 
As an application, 
we give alternative proofs of 
some content evaluation formulas, 
obtained by Han and Xiong very recently.
Our main tool is factorial Schur $Q$-functions.
\end{abstract}



\section{Introduction}

\subsection{Partitions}

Following to Macdonald's book \cite{Macdonald},
let us recall the basic knowledge on strict partitions.
A {\it partition} is a finite weakly-decreasing sequence
$\lambda=(\lambda_1,\lambda_2,\dots,\lambda_l)$
of positive integers.
The integer $\ell(\lambda)=l$ is the {\it length} of $\lambda$
and  $|\lambda|=\sum_{i=1}^l \lambda_i$ is  
the {\it size} of $\lambda$.
If $|\lambda|=n$, we say that $\lambda$ is a partition of $n$.

A partition $\lambda$ is said to be {\it strict} if all $\lambda_i$ are pairwise distinct,
and $\lambda$ is said to be {\it odd} if all $\lambda_i$ are odd integers.
Let $\mcal{SP}_n$ be the set of all strict partitions of $n$ and
$\mcal{OP}_n$ the set of all odd partitions of $n$.
The fact that their cardinalities coincide   is well known:
$|\mcal{SP}_n|=|\mcal{OP}_n|$.
Set
$\mcal{SP}=\bigcup_{n=0}^\infty \mcal{SP}_n$ and 
$\mcal{OP}=\bigcup_{n=0}^\infty \mcal{OP}_n$. 
For convenience, we deal with
the empty partition $\emptyset$ , which is 
the unique partition in $\mcal{SP}_0=\mcal{OP}_0$
with length $0$.

For each $\lambda \in \mcal{SP}$, 
we consider the set  
$$
S(\lambda)=\{ (i, j) \in \bZ^2 \ | \ 
1 \le i \le \ell(\lambda), \ i \le j \le 
\lambda_i+i-1\}.
$$
The set $S(\lambda)$ is usually drawn in a graphical way,
and called the {\it shifted Young diagram} of $\lambda$. 
Each element $\square =(i,j) \in S(\lambda)$ is often called 
a {\it box} of $\lambda$.

Let $\lambda, \mu$ be strict partitions such that $S(\lambda) \supset S(\mu)$.
Put $k=|\lambda| - |\mu|$.
A {\it standard tableau} of shape $S(\lambda/\mu)$ is a sequence of strict partitions
$(\lambda^{(0)}, \lambda^{(1)},\dots, \lambda^{(k)})$ such that
$\lambda^{(0)}=\mu$; $\lambda^{(k)}=\lambda$; and that
for each $i=1,2,\dots,k$, the diagram $S(\lambda^{(i)})$ is
obtained from $S(\lambda^{(i-1)})$ by adding exactly one box.
We denote by $g^{\lambda/\mu}$  the number of standard tableaux of
shape $S(\lambda/\mu)$.
We set $g^{\lambda/\mu}=0$ unless $S(\lambda) \supset S(\mu)$.
Define $g^\lambda= g^{\lambda/\emptyset}$.

\subsection{Shifted Plancherel measure} \label{subsec:SPM}

In this paper, we consider the following probability measure on 
$\mcal{SP}_n$, studied in many papers, e.g., in
\cite{Borodin1999, Ivanov1999, M2005a}.

\begin{definition}
The {\it shifted Plancherel measure} $\mathbb{P}_n$ on $\mcal{SP}_n$
is defined by
$$
\mathbb{P}_n (\lambda)= \frac{2^{n-\ell(\lambda)} 
(g^\lambda)^2}{n!}.
$$
\end{definition}

This indeed defines
a {\it probability}  since the identity \cite[Corollary 10.8]{HH}
$$
\sum_{\lambda \in \mcal{SP}_n} 2^{n-\ell(\lambda)} (g^\lambda)^2=n!
$$
holds. For example, if $n=5$ then
$\mathbb{P}_5 ((5))= \frac{16}{5!}$, 
$\mathbb{P}_5 ((4,1))= \frac{72}{5!}$, and
$\mathbb{P}_5 ((3,2))= \frac{32}{5!}$. 

For a function $\varphi$ on $\mcal{SP}$, we call
$$
\bE_n [\varphi] = \sum_{\lambda \in \mcal{SP}_n} 
\mathbb{P}_n(\lambda) \varphi(\lambda)=  \sum_{\lambda \in \mcal{SP}_n}
\frac{2^{n-\ell(\lambda)} (g^\lambda)^2}{n!} \varphi(\lambda)
$$
the {\it shifted Plancherel average} of $\varphi$.

Let $\{x_1,x_2,\dots\}$ be formal variables, then 
for each positive integer $r$, the $r$-th power-sum symmetric function is
defined by
$$
p_r(x_1,x_2,\dots)= x_1^r+ x_2^r+ \cdots.
$$
It is well known that 
the algebra of symmetric functions is generated by the family $\{p_r\}$.
We denote by $\Gamma$ the subalgebra generated by 
%the $p_{2m+1}$, where $m=0,1,2,\dots$.
$\{p_{2m+1}\}_{m=0,1,2,\dots}$.
Elements in $\Gamma$ are sometimes called {\it supersymmetric functions},
adapted to strict partitions.
The following theorem claims a polynomiality of $\bE_n[f]$
for any supersymmetric function $f$.

\begin{thm} \label{thm:ShPlPoly}
Suppose that $f$ is a supersymmetric function. Then
$$
\bE_n [f] = \sum_{\lambda \in \mcal{SP}_n}
\frac{2^{n-\ell(\lambda)} (g^\lambda)^2}{n!} f(\lambda_1,\lambda_2,\dots,\lambda_{\ell(\lambda)})
$$ 
is a polynomial function in $n$.
\end{thm}

We introduce a notation $x^{\downarrow k}$ as
$$
x^{\downarrow k}=x(x-1)(x-2) \cdots (x-k+1)
$$
for a variable $x$ and a positive integer $k$.
If $n$ is an integer with $0\le  n < k$ then $n^{\downarrow k}=0$.
We also set $x^{\downarrow 0}=1$.

In some supersymmetric functions $f$, we give the explicit expressions
of $\bE_n[f]$ as linear combinations of 
descending powers $n^{\downarrow j}$.
In fact, we will show 
\begin{align*}
\bE_n[p_3]=&
\sum_{\lambda \in \mcal{SP}_n} \frac{2^{n-\ell(\lambda)} (g^\lambda)^2}{n!}(\lambda_1^3+\lambda_2^3+\cdots
+ \lambda_{\ell(\lambda)}^3) = 3 n^{\downarrow 2}+ n, \\
\bE_n[p_5]=& \sum_{\lambda \in \mcal{SP}_n} \frac{2^{n-\ell(\lambda)} (g^\lambda)^2}{n!}(\lambda_1^5+\lambda_2^5+\cdots
+ \lambda_{\ell(\lambda)}^5) = \frac{40}{3} n^{\downarrow 3}+ 15 n^{\downarrow 2}+ n, \\
\bE_n[p_3^2]=&
\sum_{\lambda \in \mcal{SP}_n} \frac{2^{n-\ell(\lambda)} (g^\lambda)^2}{n!}(\lambda_1^3+\lambda_2^3+\cdots
+ \lambda_{\ell(\lambda)}^3)^2 =
9 n^{\downarrow 4} + 54 n^{\downarrow 3} +31 n^{\downarrow 2}+ n.
\end{align*}

\subsection{A deformation}

Fix a strict partition $\mu$ of $m$.
We define the measure $\mathbb{P}_{\mu,n}$ on $\mcal{SP}_{n+m}$ by
\begin{equation} \label{eq:def_Plan_mu}
\mathbb{P}_{\mu,n} (\lambda)= \frac{m!}{(n+m)!} 
2^{n-\ell(\lambda)+\ell(\mu)} \frac{g^\lambda}{g^\mu} g^{\lambda/\mu}
\qquad (\lambda \in \mcal{SP}_{n+m}).
\end{equation}
Note that $\mathbb{P}_{\mu,n} (\lambda)=0$ 
unless $S(\lambda) \supset S(\mu)$.
We will prove that $\mathbb{P}_{\mu,n}$ is a {\it probability}, i.e., 
$\sum_{\lambda \in \mcal{SP}_{n+m}} \mathbb{P}_{\mu,n}(\lambda)=1$.
For a function $\varphi$ on $\mcal{SP}$, define
$$
\bE_{\mu,n} [\varphi] = \sum_{\lambda \in \mcal{SP}_{n+m}} 
\mathbb{P}_{\mu,n}(\lambda) \varphi(\lambda)=  \sum_{\lambda \in \mcal{SP}_{n+m}}
\frac{m!}{(n+m)!} 
2^{n-\ell(\lambda)+\ell(\mu)} \frac{g^\lambda}{g^\mu} g^{\lambda/\mu} \varphi(\lambda).
$$
The summation $\bE_{\mu,n}[\varphi]$ is considered in \cite{HanXiong2015}.
Note that $\bE_{\emptyset,n}[\varphi]$ is nothing but $\bE_n[\varphi]$.
The following theorem is a slight extension of Theorem \ref{thm:ShPlPoly}.

\begin{thm} \label{thm:ShPlPoly_mu}
Let $\mu$ be a strict partition.
Suppose that $f$ is a supersymmetric function. Then
$\bE_{\mu,n} [f]$ is a polynomial function in $n$.
\end{thm}

\subsection{Content evaluations} \label{subsec:content_evaluations}

For each box $\square =(i,j)$ in $S(\lambda)$, we define
$c_\square=j-i$ and call it the {\it content} of $\square$.
We deal with symmetric functions evaluated by quantities $\widehat{c}_\square$, where
$$
\widehat{c}_\square = \frac{1}{2} c_\square (c_\square +1).
$$

\begin{thm} \label{thm:content_evaluation}
For any symmetric function $F$,
there exists a unique supersymmetric function $\widehat{F}$ in $\Gamma$ such that
$$
\widehat{F}(
\lambda_1,\lambda_2,\dots,\lambda_{\ell(\lambda)}
) = F \( \widehat{c}_\square : \square \in S(\lambda) \)
$$
for any $\lambda \in \mcal{SP}$.
Here $F \(  \widehat{c}_\square: \square \in S(\lambda) \)$
is the specialization of 
the symmetric function $F(x_1,x_2,\dots)$
such that the first $|\lambda|$ variables are substituted 
by $ \widehat{c}_\square$ for boxes $\square$ in $S(\lambda)$,
and all other variables by $0$.
\end{thm}

From Theorems \ref{thm:ShPlPoly_mu} and \ref{thm:content_evaluation},
we obtain the following corollary immediately.
This result was first obtained in \cite{HanXiong2015} very
recently.

\begin{corollary} \label{cor:HX_polynomiality}
Let $\mu$ be a strict partition of $m$.
For any symmetric function $F$, 
$$
\sum_{\lambda \in \mcal{SP}_{n+m}}
\frac{m!}{(n+m)!} 
2^{n-\ell(\lambda)+\ell(\mu)} \frac{g^\lambda}{g^\mu} g^{\lambda/\mu}  F \( \widehat{c}_\square : \square \in S(\lambda) \)
$$
is a polynomial function in $n$.
\end{corollary}

In \cite{HanXiong2015}, the reason why they
consider the quantity $F(\widehat{c}_\square:
\square \in S(\lambda))$ is not presented.
We note that 
the multi-set $\{ \widehat{c}_\square \ | \ \square \in S(\lambda)\}$
forms the collection of squares of eigenvalues  
with respect to projective analogs of Jucys--Murphy elements, see \cite{Nazarov, Sergeev, TW,
VS}
and also \cite[Theorem 3.2]{Borodin1999}.

We give the explicit expressions of 
some content evaluations.
In fact, we will show
\begin{align*} 
\mathbb{E}_n[\widehat{p_2}] =&\sum_{\lambda \in \mcal{SP}_{n}}
\frac{ 2^{n-\ell(\lambda)} (g^\lambda)^2}{n!} 
\sum_{\square \in S(\lambda)}
(\widehat{c}_\square)^2 
= \frac{2}{3} n^{\downarrow 3} + \frac{1}{2} n^{\downarrow 2}, \\
\mathbb{E}_n[\widehat{p}_{(1^2)}]=&
\sum_{\lambda \in \mcal{SP}_{n}}
\frac{ 2^{n-\ell(\lambda)} (g^\lambda)^2}{n!} 
\left\{\sum_{\square \in S(\lambda)}
\widehat{c}_\square \right\}^2
= \frac{1}{12} (n^{\downarrow 4} + 4 n^{\downarrow 3}
-8 n^{\downarrow 2} -2n).
\end{align*}
Furthermore, we will give a new algebraic proof of the identity
$$
\bE_{\mu,n}[ \widehat{p}_1 - \widehat{p}_1(\mu)] 
= \frac{1}{2} n^{\downarrow 2} + |\mu| n,
$$
which is given in \cite[Theorem 1.3]{HanXiong2015}.

\subsection{Related research and the aim} \label{subsec:related}

Let $\mcal{P}_n$ be the set of all (not necessary strict)
partitions of $n$. 
The (traditional) {\it Plancherel probability measure} 
$\mathbb{P}_n^{\mathrm{Plan}}$
on $\mcal{P}_n$ is defined by
$$
\mathbb{P}_n^{\mathrm{Plan}} (\lambda)= 
\frac{(f^\lambda)^2}{n!},
$$
where $f^\lambda$ is the number of standard tableaux of shape $Y(\lambda)$.
Here $Y(\lambda)$ is the ordinary Young diagram of $\lambda$:
$Y(\lambda)=\{(i,j) \in \bZ^2 \ | \ 1 \le i \le \ell(\lambda), \ 
1 \le j \le \lambda_i\}$.
Let $F$ be a symmetric function.
In \cite{Stanley2010} (see also \cite{Han2009}), Stanley proves that
the summation 
$$
\sum_{\lambda \in \mcal{P}_n}\mathbb{P}_n^{\mathrm{Plan}} (\lambda) F(h_\square^2 : \square \in Y(\lambda))
$$
is a polynomial in $n$.
Here $h_\square$ denotes the hook length of the square $\square$
in the Young diagram $Y(\lambda)$.
Panova \cite{Panova2012} shows 
an explicit identity for 
%particular $F$.
the symmetric function $F=F_r(x_1,x_2,\dots)=
\sum_{j \ge 1} \prod_{i=1}^r (x_j-i^2)$.

Moreover, Stanley \cite{Stanley2010} proves that 
the content evaluation
\begin{equation} \label{eq:Plan_ave_content}
\sum_{\lambda \in \mcal{P}_n}\mathbb{P}_n^{\mathrm{Plan}} (\lambda) F(c_\square : \square \in Y(\lambda))
\end{equation}
is also a polynomial in $n$.
Olshanski \cite{Olshanski2010} finds that 
the functions $\lambda \mapsto F(c_\square : \square \in Y(\lambda))$
are seen as shifted-symmetric functions in variables 
$\lambda_1,\lambda_2,\dots$ and obtains an alternative 
algebraic proof for the polynomiality of \eqref{eq:Plan_ave_content}.
Some explicit formulas for particular $F$ are obtained in \cite{Feray2012, FKMO, LT, Lassalle2013, M2011, MN}.
Just as an example, in \cite{FKMO} the identity
$$
\sum_{\lambda \in \mcal{P}_n} \mathbb{P}_n^{\mathrm{Plan}} (\lambda)
\sum_{\square \in Y(\lambda)}  \prod_{i=0}^{k-1} (c_\square^2 -i^2)=
\frac{(2k)!}{ ((k+1)!)^2} n^{\downarrow (k+1)}
$$
is obtained.

We emphasize the fact that the content evaluations
are related to matrix integrals (\cite{M2011, MN}).
For example, 
let us consider the unitary group $U(N)$ with the normalized Haar measure $dU$
and suppose $n \le N$.
Then Weingarten calculus gives the following identity 
\begin{align*}
&\int_{U(N)} |u_{1 1} u_{22} \cdots u_{nn}|^2 \, dU \\
=&
\sum_{k=0}^\infty (-1)^k N^{-(n+k)} \sum_{\lambda \in \mcal{P}_n}
\mathbb{P}_n^{\mathrm{Plan}} (\lambda) h_k (c_{\square}: \ \square \in Y(\lambda)).
\end{align*}
Here $h_k$ are complete symmetric functions.
More general identities (for other classical groups) 
can be seen in \cite{M2011, MN}.

The quantity $F(\widehat{c}_\square: \square \in S(\lambda))$ in Subsection \ref{subsec:content_evaluations}
is a natural projective analog
of $F(c_\square: \square \in Y(\lambda))$,
because the $c_{\square}$ are eigenvalues of
 Jucys--Murphy elements of the symmetric groups,
while the $\widehat{c}_\square$ come from
their projective version.
Unfortunately, 
%the projective content evaluation 
% $F(\widehat{c}_\square: \square \in S(\lambda))$
%does not have any connection with matrix integrals.
it is not known any direct connection between matrix integrals and
the projective content evaluation 
$F(\widehat{c}_\square: \square \in S(\lambda))$.

Our results in this paper are seen as the counterparts of the content evaluation
\cite{Olshanski2010} in the theory of the shifted Plancherel measure.
As Olshanski does in \cite{Olshanski2010}, we employ factorial versions of
symmetric functions.
Specifically, we introduce 
a new family of 
supersymmetric functions
$(\mf{p}_{\rho})_{\rho \in \mcal{OP}}$.
The function $\mf{p}_\rho$ is also regarded as 
projective (or spin) irreducible character values 
of the symmetric groups.
For ordinary partitions, the counterpart is 
the normalized linear character, which has been studied
in e.g. \cite{CGS2004, DolegaFeray, Feray2010, IvanovKerov, Stanley2006},
and written as $p^{\#}_\rho$,
$\widehat{\chi}_\rho$, 
$\mathrm{Ch}_\rho$, $\dots$ in their articles.
We will provide explicit values of shifted Plancherel averages 
$\bE_{\mu,n}[\mf{p}_\rho]$
for all strict partitions $\mu$ and odd partitions $\rho$.

As mentioned above, Corollary \ref{cor:HX_polynomiality}
is obtained by Han and Xiong \cite{HanXiong2015}.
Our purpose in this paper is to provide more insight for their result,
based on the theory of factorial Schur $Q$-functions.
As a result, we can obtain some new identities given in Subsections \ref{subsec:SPM}
and \ref{subsec:content_evaluations} in a simple way.

\subsection{Outline of the paper}

The paper is organized as follows.
Section \ref{section:supersymmetric_functions} gives definitions and 
basis properties of Schur $P$- and factorial Schur $P$-functions.
A more detailed description can be seen in \cite{Ivanov1999, Ivanov2005, Macdonald}.
In Section \ref{sec:NewSupersymmetricFunctions} we introduce
new supersymmetric functions $\mf{p}_\rho$ and provide some necessary properties.
In Section \ref{sec:ShPlAv_Proof} we give the proofs of 
Theorem \ref{thm:ShPlPoly} and Theorem \ref{thm:ShPlPoly_mu}. New identities presented in Subsection \ref{subsec:SPM} are also proved.
In Section \ref{sec:Content_Evaluation} we give a proof of Theorem \ref{thm:content_evaluation} and present some examples of 
content evaluations.
In Section \ref{sec:Remark_HX} we deal with 
some family of functions on $\mcal{SP}$
introduced in \cite{HanXiong2015}
and show that they are supersymmetric functions.
We comment on some remaining questions in Section \ref{sec:OpenProblems}.

\section{Supersymmetric functions} \label{section:supersymmetric_functions}

\subsection{The algebra of supersymmetric functions}

A {\it symmetric function} is a collection of
polynomials $F=(F_N)_{N =1,2,\dots}$ with rational coefficients
such that
\begin{itemize}
\item each $F_N$ is symmetric in
$N$ commutative variables $x_1,x_2, \dots,x_N$;
\item the stability relation
$F_{N+1}(x_1,\dots,x_N,0)=F_N(x_1,\dots,x_N)$
holds for all $N \ge 1$.
\end{itemize}
We often write $F$ as $F(x_1,x_2,\dots)$ in infinitely-many variables $x_1,x_2,\dots$.

For each $r=1,2,\dots,$ the $r$-th power-sum symmetric function $p_r$
is given by
$$
p_r(x_1,x_2,\dots)=x_1^r+x_2^r+ \cdots.
$$
It is well known that the $p_r$ generate the algebra of 
all symmetric functions and are algebraically independent over $\bQ$.

\begin{definition}
Let $\Gamma$ denote the subalgebra of symmetric functions generated by 
$p_{2m+1}$, $m=0,1,2,\dots$.
We say elements  in $\Gamma$ to be {\it supersymmetric functions}.
\end{definition}

Let us review supersymmetric functions along 
\cite[Chapter III.8]{Macdonald}.
Define
$$
p_{\rho}= p_{\rho_1} p_{\rho_2} \cdots p_{\rho_l}
$$
for $\rho=(\rho_1,\rho_2,\dots,\rho_l) \in \mcal{OP}$.
The $p_\rho$ form a linear basis of $\Gamma$ by definition.
The scalar product on $\Gamma$ is defined by
\begin{equation} \label{eq:scalar_p}
\langle p_\rho, p_{\sigma} \rangle = 2^{-\ell(\rho)} z_\rho \delta_{\rho \sigma}
\end{equation}
for $\rho, \sigma \in \mcal{OP}$, where
$$
z_\rho= \prod_{r \ge 1} r^{m_r(\rho)} m_r(\rho)!,
$$
and $m_r(\rho)$ is the multiplicity of $r$ in $\rho$:
$m_r(\rho)= |\{i \in \{1,2,\dots, \ell(\rho)\} \ | \ \rho_i=r\}|$.

%The following lemma is seen in 
%\cite[III.8, Example 11]{Macdonald}.
%
%\begin{lemma} \label{lem:p1dual}
%For $f, g \in \Gamma$, 
%$$
%\langle p_1 f, g \rangle = \frac{1}{2} 
%\left\langle f, \frac{\partial}{\partial p_1} g \right\rangle.
%$$
%Here the differential operator 
%$\frac{\partial}{\partial p_1}$ acts on functions in $\Gamma$ 
%expressed as polynomials in $p_1,p_3,p_5,\dots$.
%\end{lemma}

Given an element $f$ in $\Gamma$ and a strict partition $\lambda$,
we denote by $f(\lambda)$ the value
$$
f(\lambda_1,\lambda_2,\dots,\lambda_{\ell(\lambda)},0,0,\dots).
$$
For example, 
$p_3(\lambda)= \lambda_1^3+\lambda_2^3+ \cdots + \lambda_{\ell(\lambda)}^3$.
Elements in $\Gamma$ are uniquely defined by their values
on $\mcal{SP}$, i.e., 
two elements $f,g$ in $\Gamma$ coincide with each other 
if and only if it holds that $f(\lambda) = g(\lambda)$ 
for every strict partition $\lambda$.

\subsection{Schur $P$-functions}

Let us review the Schur $P$-function, which
is the particular $t=-1$ case of the Hall--Littlewood function
with parameter $t$.
We use the definition in \cite{Ivanov1999}.
See also \cite[III.8]{Macdonald} and  \cite{HH} for details.

\begin{definition}
Let $\lambda=(\lambda_1,\lambda_2,\dots,\lambda_l)
 \in \mcal{SP}$.
Suppose that $N \ge l= \ell(\lambda)$.
We define a polynomial $P_{\lambda |N}$ by 
$$
P_{\lambda |N}(x_1,\dots,x_N) 
= \frac{1}{(N-l)!} \sum_{\omega \in S_N} \omega \Big(
x_1^{\lambda_1}
x_2^{\lambda_2} \cdots x_l^{\lambda_l}
\prod_{\begin{subarray}{c} i: 1 \le i \le l, \\
j: i<j \le N \end{subarray}}
\frac{x_i+x_j}{x_i-x_j}
 \Big),
$$
where the symmetric group $S_N$ acts by permuting the variables 
$x_1,\dots,x_N$.
If $\ell(\lambda)>N$, then we set $P_{\lambda |N}(x_1,\dots,x_N)=0$.
Then the collection $(P_{\lambda |N})_{N =1,2,\dots}$ defines
an element $P_{\lambda}$ in $\Gamma$.
Define $Q_{\lambda}$ by $Q_{\lambda}=2^{\ell(\lambda)} P_{\lambda}$.
We call $P_\lambda$ and $Q_{\lambda}$ 
a {\it Schur $P$-function} and {\it Schur $Q$-function}, respectively.
\end{definition}

\begin{prop} \label{prop:scalar_product}
\begin{enumerate}
\item[(i)] The family $(P_\lambda)_{\lambda \in \mcal{SP}}$
forms a linear basis of $\Gamma$.
\item[(ii)] $\langle P_\lambda, Q_\mu \rangle = \delta_{\lambda \mu}$
for $\lambda, \mu \in \mcal{SP}$.
\item[(iii)] $\sum_{\lambda \in \mcal{SP}_n} 2^{-\ell(\lambda)} 
\langle f, Q_\lambda \rangle \langle g, Q_\lambda \rangle
= \langle f, g \rangle$ for $f,g \in \Gamma$.
\item[(iv)] Suppose $|\lambda| \ge |\mu|$. Then 
$g^{\lambda/\mu} =\langle p_1^{|\lambda|-|\mu|} P_\mu, Q_\lambda \rangle$
for $\lambda, \mu \in \mcal{SP}$,
where $g^{\lambda/\mu}$ is the number of
standard tableaux of shape $S(\lambda/\mu)$. 
In particular,
$g^{\lambda} = g^{\lambda/\emptyset}=\langle p_1^{|\lambda|}, Q_\lambda \rangle$.
\end{enumerate}
\end{prop}

\begin{proof}
(i), (ii): See 
\cite[Chapter III, (8.9) and (8.12)]{Macdonald}.

(iii): 
By linearity, it is enough to show the identity for $f=P_\mu$ and 
$g=Q_\nu$ with $|\mu|=|\nu|=n$.
Then both sides are equal to $\delta_{\mu \nu}$ by Claim (ii).

(iv): We recall a special case of 
the Pieri-type formula for Schur $P$-functions
(\cite[Chapter III, (8.15)]{Macdonald}): 
for $\mu \in \mcal{SP}$,
$$
p_1 P_\mu =\sum_{\mu^+ :\mu^+ \searrow \mu} P_{\mu^+},
$$
where the sum runs over strict partitions $\mu^+$ obtained from
$\mu$ by adding one box.
Put $k=|\lambda|-|\mu|$.
The integer $g^{\lambda/\mu}$ is the number of
sequences $(\lambda^{(0)},\lambda^{(1)},\dots, \lambda^{(k)})$ 
of strict partitions such that
$\lambda^{(0)}=\mu$, $\lambda^{(k)}=\lambda$, and such that
$\lambda^{(i)} \searrow \lambda^{(i-1)}$ for each $i=1,2,\dots, k$. 
Therefore we find the formula 
$$
p_1^{k} P_\mu = \sum_{\lambda \in \mcal{SP}_{|\mu|+k}
} g^{\lambda/\mu}
P_\lambda.
$$
Claim (iv) now follows from Claim (ii).
\end{proof}

For a strict partition $\lambda$ and odd partition $\rho$ of sizes $k$, 
we define
\begin{equation} \label{eq:def_X}
X^\lambda_\rho= \langle p_\rho, Q_\lambda \rangle.
\end{equation}
Equivalently, the quantities $X^\lambda_\rho$ are determined as transition matrices via
\begin{equation} \label{eq:trans_p_P}
p_\rho= \sum_{\lambda \in \mcal{SP}_k} X^\lambda_\rho P_\lambda
\qquad \text{or} \qquad 
Q_\lambda= \sum_{\rho \in \mcal{OP}_k}
 2^{\ell(\rho)}z_\rho^{-1} X^\lambda_\rho p_\rho. 
\end{equation}
The quantity $X^\lambda_\rho$ is a character value
for a projective representation of symmetric groups,
see \cite[Chapter 8]{HH}.
One can compute values $X^\lambda_\rho$ recursively 
if we use a Murnaghan--Nakayama rule 
(\cite[Chapter III.8, Example 11]{Macdonald}).
%In \cite{Macdonald}, 
%the $X^\lambda_\rho$ is written as $X^\lambda_\rho(-1)$.
Note that $X^{\lambda}_{(1^{k})}=g^\lambda$ and $X^{(k)}_\rho =1$. Equivalently, 
$$
p_{1}^k = \sum_{\lambda \in \mcal{SP}_k} g^\lambda P_\lambda
\qquad \text{and} \qquad 
Q_{(k)}= \sum_{\rho \in \mcal{OP}_k} 2^{\ell(\rho)} z_{\rho}^{-1} p_\rho.
$$

\begin{prop} \label{prop:OrthoXX}
For $\rho, \sigma \in \mcal{OP}_k$,
$$
\sum_{\lambda \in \mcal{SP}_k} 2^{-\ell(\lambda)} X^\lambda_\rho X^\lambda_{\sigma}
=\delta_{\rho, \sigma} 2^{-\ell(\rho)} z_{\rho}.
$$
\end{prop}

\begin{proof}
It follows from \eqref{eq:trans_p_P} and 
Proposition \ref{prop:scalar_product} (ii) that 
$$
\langle p_\rho, p_{\sigma} \rangle  
=
\langle \sum_\lambda X^\lambda_\rho P_\lambda, \sum_\mu X^\mu_\sigma P_\mu \rangle \\
= \sum_{\lambda, \mu} X^\lambda_\rho X^\mu_\sigma \langle P_\lambda, P_\mu \rangle 
= \sum_{\lambda}  X^\lambda_\rho X^\lambda_\sigma 2^{-\ell(\lambda)},
$$
the left hand side of which equals
$2^{-\ell(\rho)} z_{\rho} \delta_{\rho \sigma} $ by \eqref{eq:scalar_p}.
\end{proof}

\subsection{Factorial Schur $P$-functions}

The next definition is due to A. Okounkov and 
given in \cite{Ivanov1999}.

\begin{definition} \label{def:factrialP}
Let $\lambda=(\lambda_1,\dots,\lambda_l) \in \mcal{SP}$.
Suppose that $N \ge l= \ell(\lambda)$.
We introduce a polynomial $P_{\lambda |N}^*$ by 
$$
P_{\lambda |N}^*(x_1,\dots,x_N) 
= \frac{1}{(N-l)!} \sum_{\omega \in S_N} \omega \Big(
x_1^{\downarrow \lambda_1}
x_2^{\downarrow \lambda_2} \cdots x_l^{\downarrow \lambda_l}
\prod_{\begin{subarray}{c} i: 1 \le i \le l, \\
j: i<j \le N \end{subarray}}
\frac{x_i+x_j}{x_i-x_j}
 \Big).
$$
If $\ell(\lambda)>N$, then we set $P_{\lambda |N}^*(x_1,\dots,x_N)=0$.
The collection $(P_{\lambda |N}^*)_{N =1,2,\dots}$ defines
an element $P_{\lambda}^*$ in $\Gamma$.
Define $Q_{\lambda}^*$ by $Q_{\lambda}^*=2^{\ell(\lambda)} P_{\lambda}^*$.
We call $P_\lambda^*$ and $Q_{\lambda}^*$ 
the {\it factorial Schur $P$-function} and {\it $Q$-function}, respectively.
\end{definition}

Remark that $P_\lambda$ is homogeneous, whereas
$P_\lambda^*$ is not.
Let us review some properties for factorial Schur $P$-functions.
See \cite{Ivanov1999, Ivanov2005} for detail.
We note that Claims (i)--(iii) in the next proposition
are immediately comfirmed from
definitions, while the proof of claim (iv) requires a more careful work.

\begin{prop} \label{prop:property_factorial}
\begin{enumerate}
\item[(i)] $P_\lambda^*= P_\lambda+g$,
where $g$ is a supersymmetric function of degree less than $|\lambda|$.
\item[(ii)] The family $(P_\lambda^*)_{\lambda \in \mcal{SP}}$
forms a linear basis of $\Gamma$.
\item[(iii)] 
$P_\mu^*(\lambda)=0$ unless $S(\mu) \subset S(\lambda)$.
\item[(iv)] It holds that 
\begin{equation} \label{eq:P*_g}
P_\mu^*(\lambda)= |\lambda|^{\downarrow |\mu|} \frac{g^{\lambda/\mu}}{g^\lambda}
=  |\lambda|^{\downarrow |\mu|} \frac{\langle p_1^{|\lambda|-|\mu|} P_\mu,Q_\lambda
\rangle}{g^\lambda}.
\end{equation}
\end{enumerate}
\end{prop}

On the right hand side of \eqref{eq:P*_g}, we can think $|\lambda| -|\mu|$ being nonnegative, because 
if $|\lambda| <|\mu|$ then 
$|\lambda|^{\downarrow |\mu|}=0$.

Next we give a formula for 
an expansion of $P_\lambda$ in terms of $P_\mu^*$.
Recall the Stirling numbers $T(k,j)$ of the second kind
defined by
\begin{equation} \label{eq:stirling_2}
x^{k} = \sum_{j=1}^k T(k,j) x^{\downarrow j} \qquad (k=1,2,\dots).
\end{equation}

\begin{prop} \label{prop:PtoP*_Stirling}
Let $\lambda$ be a strict partition of length $l$.
Then
$$
P_\lambda= \sum_{j_1=1}^{\lambda_1} \cdots \sum_{j_l=1}^{\lambda_l}
T(\lambda_1,j_1) \cdots T(\lambda_l,j_l) P^*_{(j_1,\dots,j_l)}.
$$
Here we set 
$P^*_{(j_1,\dots,j_l)}=0$ if 
$j_1,\dots, j_l$ are not pairwise distinct, and 
$$
P^*_{(j_1,\dots,j_l)}= (\sgn \pi) P_{\mu}^*
$$
if $(j_1,\dots,j_l)=(\mu_{\pi(1)},\dots, \mu_{\pi(l)})$
for some strict partition $\mu=(\mu_1,\dots,\mu_l)$ of length $l$ with
a permutation $\pi$. 
\end{prop}

\begin{proof}
The definition of $P^*_{\lambda|N} (x_1,\dots,x_N)$ in 
Definition \ref{def:factrialP} makes sense even if 
$(\lambda_1,\dots, \lambda_l)$ is replaced with  
any sequence of positive integers $(j_1,\dots,j_l)$.
From the alternating property
$$
\pi \Big( \prod_{\begin{subarray}{c} i: 1 \le i \le l, \\
j: i<j \le N \end{subarray}}
\frac{x_i+x_j}{x_i-x_j} \Big) = (\sgn \pi) \prod_{\begin{subarray}{c} i: 1 \le i \le l, \\
j: i<j \le N \end{subarray}}
\frac{x_i+x_j}{x_i-x_j}
$$
for a permutation $\pi \in S_l$ acting on variables $x_1,\dots,x_l$,
we have 
$$
P^*_{(j_{\pi(1)},\dots,j_{\pi(l)})|N}(x_1,\dots,x_N) =
(\sgn \pi) P^*_{(j_1,\dots,j_l)|N}(x_1,\dots,x_N).
$$
In particular, 
$P^*_{(j_1,\dots,j_l)}=0$ if $j_s=j_t$ for some $s \not=t$.
Our proposition follows from the
definitions of $P_{\lambda}, P^*_\mu$, and  Stirling numbers.
\end{proof}

\section{New supersymmetric functions} \label{sec:NewSupersymmetricFunctions}

Recall the fact that
two families 
 $(P_\lambda)_{\lambda \in \mcal{SP}}$ and  
 $(P_\lambda^*)_{\lambda \in \mcal{SP}}$
 are liner bases of $\Gamma$.
Define a linear isomorphism $\Psi: \Gamma \to \Gamma$ by
$$
\Psi (P_\lambda)= P_{\lambda}^* \qquad (\lambda \in \mcal{SP}).
$$
Note that $\Psi^{-1} (f)$ coincides with the  top-degree term of $f$
by Proposition \ref{prop:property_factorial} (i).

\begin{definition}
For each $\rho \in \mcal{OP}_k$, we define the supersymmetric function 
$\mf{p}_\rho$ by
$\mf{p}_\rho = \Psi (p_\rho)$. From \eqref{eq:trans_p_P}, we have 
$$
\mf{p}_\rho = \sum_{\lambda \in \mcal{SP}_k} X^\lambda_\rho P_\lambda^*. 
$$ 
\end{definition}

For an odd partition $\rho$, we denote by $\tilde{\rho}$
the odd partition obtained from $\rho$ by erasing parts equal to $1$.
For example,
if $\rho=(5,5,3,1,1)$, then $\tilde{\rho}=(5,5,3)$.
Note that
$|\rho|=|\tilde{\rho}|+m_1(\rho)$ and 
$\ell(\rho)= \ell(\tilde{\rho})+m_1(\rho)$.

\begin{prop} \label{prop:mp}
\begin{enumerate}
\item[(i)] $\mf{p}_\rho= p_{\rho} + g$,
where $g$ is a supersymmetric function of degree less than $|\rho|$.
\item[(ii)] The family $(\mf{p}_\rho)_{\rho \in \mcal{OP}}$ forms a
linear basis of $\Gamma$.
\item[(iii)]   
For $\rho \in \mcal{OP}$ and $\lambda \in \mcal{SP}$, 
$$
\mf{p}_{\rho} (\lambda) = |\lambda|^{\downarrow |\rho|} \frac{\langle 
p_1^{|\lambda|-|\tilde{\rho}|} p_{\tilde{\rho}}, Q_\lambda \rangle}{g^\lambda}
=|\lambda|^{\downarrow |\rho|} \frac{
X^{\lambda}_{\tilde{\rho} \cup (1^{|\lambda|-|
\tilde{\rho}|})}}{g^\lambda}.
$$
Here $\tilde{\rho} \cup (1^{k})$ denotes the odd partition
$(\tilde{\rho},\underbrace{1,1,\dots,1}_k)$.
\item[(iv)] For $\rho \in \mcal{OP}$ and $\lambda \in \mcal{SP}$, 
$$
\mf{p}_\rho(\lambda)= (|\lambda|-|\tilde{\rho}|)^{\downarrow m_1(\rho)}
\mf{p}_{\tilde{\rho}} (\lambda).
$$
\item[(v)] If $|\tilde{\rho}| > |\lambda|$, then
$\mf{p}_\rho(\lambda)=0$.
\end{enumerate}
\end{prop}

\begin{proof}
(i): It follows immediately from \eqref{eq:trans_p_P} and 
Proposition \ref{prop:property_factorial} (i).

(ii): It follows immediately from  Claim (i) and 
the fact that $(p_{\rho})_{\rho \in \mcal{OP}}$ form a linear basis
of $\Gamma$.

(iii): Set $k=|\rho|$.
From Proposition \ref{prop:property_factorial} (iv),  \eqref{eq:trans_p_P}, 
and \eqref{eq:def_X}, we have
\begin{align*}
\mf{p}_{\rho} (\lambda) =& \sum_{\mu \in \mcal{SP}_k} X^\mu_\rho P^*_\mu(\lambda)=
\sum_{\mu \in \mcal{SP}_k} X^\mu_\rho
|\lambda|^{\downarrow k} \frac{\langle 
p_1^{|\lambda|-k} P_\mu, Q_\lambda \rangle}{g^\lambda} \\
=& |\lambda|^{\downarrow k} \frac{\langle 
p_1^{|\lambda|-k} p_{\rho}, Q_\lambda \rangle}{g^\lambda}
= |\lambda|^{\downarrow k} \frac{
X^{\lambda}_{\rho \cup (1^{|\lambda|-k})}}{g^\lambda}.
\end{align*}
Note that $\rho \cup (1^{|\lambda|-k})
=\tilde{\rho} \cup(1^{|\lambda|- |\tilde
{\rho}|})$.

(iv): Since  
$|\lambda|^{\downarrow |\rho|} =
|\lambda|^{\downarrow |\tilde{\rho}|}
\cdot (|\lambda|-|\tilde{\rho}|)^{
\downarrow m_1(\rho)}$, 
Claim (iii) implies that
$$
\mf{p}_{\rho} (\lambda) =
|\lambda|^{\downarrow |\tilde{\rho}|}
\cdot (|\lambda|-|\tilde{\rho}|)^{
\downarrow m_1(\rho)}
\frac{
X^{\lambda}_{\tilde{\rho} \cup(1^{|\lambda|- |\tilde
{\rho}|})}}{g^\lambda}
= (|\lambda|-|\tilde{\rho}|)^{
\downarrow m_1(\rho)}
\mf{p}_{\tilde{\rho}} (\lambda).
$$

(v): If $|\lambda| < |\tilde{\rho}|$,
then
$
|\lambda|^{\downarrow |\tilde{\rho}|} =
|\lambda|(|\lambda|-1) \cdots
(|\lambda|-|\tilde{\rho}|+1)=0$,
and therefore we obtain
$\mf{p}_{\rho}(\lambda)=0$ 
from Claim (iii).
\end{proof}

Substituting
$\rho=(1^k)$ in Proposition \ref{prop:mp} (iv),
we obtain
$\mf{p}_{(1^k)}(\lambda) =|\lambda|^{\downarrow k}$.
Using Stirling numbers defined in \eqref{eq:stirling_2}, we find
$$
p_{(1^k)}= \sum_{j=1}^k T(k,j) \mf{p}_{(1^j)}.
$$
More generally,
a power-sum function $p_{\rho}$ can be expanded as a linear combination
of $\mf{p}_\sigma$ 
in the following way.
First, we expand $p_{\rho}$ in terms of $P_{\lambda}$ by using \eqref{eq:trans_p_P}.
Second, each $P_\lambda$ is expanded in terms of
factorial Schur $P$-functions $P_\mu^*$ by Proposition \ref{prop:PtoP*_Stirling}.
Finally, each $P_\mu^*$ is expanded in terms of $\mf{p}_{\sigma}$ by the formula
$$
P_{\mu}^*= \sum_{\sigma \in \mcal{OP}_{|\mu|}} 2^{-\ell(\mu)+\ell(\sigma)}
z_{\sigma}^{-1} X^{\mu}_\sigma \mf{p}_\sigma,
$$
which is the image of the second equation on \eqref{eq:trans_p_P}
under $\Psi$.

\begin{example} \label{example:p_to_mp}
\begin{align*}
p_{(1)}=& \mf{p}_{(1)}, \\
p_{(1^2)}=& \mf{p}_{(1^2)}+\mf{p}_{(1)}, \\
p_{(3)}=&\mf{p}_{(3)} +3 \mf{p}_{(1^2)} +\mf{p}_{(1)}, \\
p_{(1^3)}=&\mf{p}_{(1^3)} +3 \mf{p}_{(1^2)} +\mf{p}_{(1)} \\
p_{(3,1)}=& \mf{p}_{(3,1)} +3 \mf{p}_{(3)} +3 \mf{p}_{(1^3)} 
+7 \mf{p}_{(1^2)} +\mf{p}_{(1)}, \\
p_{(1^4)}=&\mf{p}_{(1^4)} +6 \mf{p}_{(1^3)} +7 \mf{p}_{(1^2)} +  \mf{p}_{(1)}, \\ 
p_{(5)}=& \mf{p}_{(5)} + 10 \mf{p}_{(3,1)} + \frac{35}{3} \mf{p}_{(3)} + \frac{40}{3} \mf{p}_{(1^3)}
+ 15 \mf{p}_{(1^2)} +\mf{p}_{(1)}, \\
p_{(3,1,1)}=& \mf{p}_{(3,1,1)} + 7 \mf{p}_{(3,1)} +3 \mf{p}_{(1^4)} + 
9 \mf{p}_{(3)} +16 \mf{p}_{(1^3)}
+15 \mf{p}_{(1^2)} +\mf{p}_{(1)}, \\
p_{(1^5)} =& \mf{p}_{(1^5)} +10 \mf{p}_{(1^4)} 
+25 \mf{p}_{(1^3)} + 15 \mf{p}_{(1^2)} +\mf{p}_{(1)}.
\end{align*}
\end{example}

\begin{remark}
For two ordinary partitions $\lambda, \mu$, we consider 
$$
\mathrm{Ch}_\mu (\lambda)=
\begin{cases}
|\lambda|^{\downarrow |\mu|} \frac{\chi^\lambda_{\mu \cup (1^{|\lambda|-|\mu|})}}
{f^\lambda} &  \text{if $|\lambda| \ge |\mu|$}, \\
0 & \text{if $|\lambda|<|\mu|$},
\end{cases}
$$
where $\chi^\lambda_{\mu \cup (1^{|\lambda|-|\mu|})}$
is the value
of the irreducible character $\chi^\lambda$
of the symmetric group $S_{|\lambda|}$
at conjugacy class associated with $\mu \cup (1^{|\lambda|-|\mu|})$.
The functions $\mathrm{Ch}_\mu$ on the set of all partitions are called
the normalized characters of  symmetric groups, and 
have rich properties and applications.
See \cite{Feray2010, IO2002, Stanley2006}.
Note that the function is written as $p_\mu^{\#}$ in \cite{IO2002}.
Our function $\mf{p}_\rho$ is
a {\it projective} analog of $\mathrm{Ch}_\mu$
since 
$X^\lambda_\rho$ is a character value
for a projective representation of symmetric groups.
\end{remark}

\section{Shifted Plancherel averages} \label{sec:ShPlAv_Proof}

\subsection{Proof of Polynomiality}

In the present section
we give a proof of 
Theorems  \ref{thm:ShPlPoly} and  \ref{thm:ShPlPoly_mu}.
Let $m$ be a nonnegative integer.
Fix $\mu \in \mcal{SP}_m$. For each $\rho \in \mcal{OP}$,
we consider the summation
$$
\bE_{\mu,n}[\mf{p}_{\rho}] = \sum_{\lambda \in \mcal{SP}_{n+m}}
\frac{m!}{(n+m)!} 2^{n-\ell(\lambda) +\ell(\mu)}
\frac{g^\lambda}{g^\mu} g^{\lambda/\mu} \mf{p}_{\rho} (\lambda).
$$
Since Proposition \ref{prop:mp} (iv) implies that
\begin{equation} \label{eq:mu_average_M_1}
\bE_{\mu,n} [\mf{p}_{\rho}] = (n+m-|\tilde{\rho}|)^{ \downarrow m_1(\rho)}
\bE_{\mu,n} [\mf{p}_{\tilde{\rho}}],
\end{equation}
it is sufficient to compute 
$\bE_{\mu,n}[\mf{p}_{\rho}]$
for odd partitions $\rho$ with no part equal to $1$.

The following lemma is seen in 
\cite[III.8, Example 11]{Macdonald}.

\begin{lemma} \label{lem:p1dual}
For $f, g \in \Gamma$, 
$$
\langle p_1 f, g \rangle = \frac{1}{2} 
\left\langle f, \frac{\partial}{\partial p_1} g \right\rangle.
$$
Here the differential operator 
$\frac{\partial}{\partial p_1}$ acts on functions in $\Gamma$ 
expressed as polynomials in $p_1,p_3,p_5,\dots$.
\end{lemma}

\begin{thm} \label{thm:mu_average_M}
Let $\rho$ be an odd partition such that $m_1(\rho)=0$.
Then 
\begin{equation} \label{eq: mu_average_M_M}
\bE_{\mu,n}[\mf{p}_{\rho}]= \mf{p}_{\rho}(\mu).
\end{equation}
\end{thm}

\begin{proof}
First of all,
if $n+m  < |\rho|$, then we have
$\mf{p}_{\rho}(\lambda)=0$ for all
$\lambda \in \mcal{SP}_{n+m}$ and 
$\mf{p}_{\rho}(\mu)=0$ by virtue of Proposition \ref{prop:mp} (v), so that
$\bE_{\mu,n}[\mf{p}_\rho]=0 =\mf{p}_\rho(\mu)$.
Consequently, we may assume $|\rho| \le n+m$.

Using Proposition \ref{prop:scalar_product} (iv) and 
Proposition \ref{prop:mp} (iii), 
we see that
\begin{align}
 \bE_{\mu,n}[\mf{p}_{\rho}] 
=& \sum_{\lambda \in \mcal{SP}_{n+m}} \frac{m!}{ (n+m)!} 
2^{n-\ell(\lambda)+\ell(\mu)} \frac{g^\lambda}{g^\mu} 
\langle p_1^n P_\mu, Q_\lambda \rangle \notag \\
& \qquad \qquad \times  (n+m)^{\downarrow |\rho|} 
\frac{\langle p_1^{n+m-|\rho|} p_{\rho}, Q_\lambda \rangle}
{g^\lambda}
 \notag \\
=& 
\frac{m!}{g^\mu} \frac{(n+m)^{\downarrow |\rho|}}{(n+m)!}  \sum_{\lambda \in \mcal{SP}_{n+m}}
2^{n-\ell(\lambda)}\langle p_1^n Q_\mu, Q_\lambda \rangle
\langle p_1^{n+m-|\rho|} p_{\rho}, Q_{\lambda} \rangle  \notag \\
=& \frac{m!}{g^\mu}  \frac{(n+m)^{\downarrow |\rho|}}{(n+m)!} 
2^{n} \langle p_1^n Q_\mu, p_1^{n+m-|\rho|} p_{\rho} \rangle,
\label{eq:E_mu_n_pp}
\end{align}
where we have used Proposition \ref{prop:scalar_product} (iii)
for the last equality.

We next compute the scalar product 
$\langle p_1^n Q_\mu, p_1^{n+m-|\rho|} p_{\rho} \rangle$.
Assume that $|\rho| >m$. 
Expanding $Q_\mu$ in $p_{\sigma}$ (see \eqref{eq:trans_p_P}),  we have 
\begin{align*}
\langle p_1^n Q_\mu, p_1^{n+m-|\rho|} p_{\rho} \rangle
=& \sum_{\sigma \in \mcal{OP}_m} 2^{\ell(\sigma)} z_{\sigma}^{-1} X^\mu_{\sigma}
\langle p_{1}^{n+m_1(\sigma)} p_{\tilde{\sigma}}, p_1^{n-(|\rho|-m)} 
p_{\rho} \rangle.
\end{align*}
Here all the scalar products of the right hand side vanish 
by \eqref{eq:scalar_p}
because $m_1(\rho)=0$ and $n+m_1(\sigma) \ge n > n-(|\rho|-m)$.
Therefore it follows from \eqref{eq:E_mu_n_pp} that
$$
\bE_{\mu,n} [\mf{p}_{\rho}]=0 \qquad \text{if $|\rho| >m$}.
$$
On the other hand, $\mf{p}_{\rho}(\mu)=0$ if $|\rho|> m$ by Proposition \ref{prop:mp} (v), hence $\bE_{\mu,n}[\mf{p}_{\rho}]= 0= \mf{p}_\rho(\mu)$ in 
that case.

We finally assume that $|\rho| \le m$.
Using Lemma \ref{lem:p1dual} we have
\begin{align*}
\langle p_1^n Q_\mu, p_1^{n+m-|\rho|} p_{\rho} \rangle =&
\frac{1}{2^n} \left\langle Q_\mu, \(\frac{\partial}
{\partial p_1}\)^n p_1^{n+m-|\rho|} p_{\rho} 
\right\rangle \\
=& \frac{(n+m-|\rho|)^{\downarrow n}}{2^n} 
\langle Q_{\mu}, p_1^{m-|\rho|} p_{\rho} \rangle \\
=& \frac{(n+m-|\rho|)^{\downarrow n}}{2^n} 
\frac{g^\mu}{m^{\downarrow |\rho|}} \mf{p}_{\rho} (\mu),
\end{align*}
where in the last equality Proposition \ref{prop:mp} (iii) is applied.
Combining this with \eqref{eq:E_mu_n_pp} gives
\begin{align*}
\bE_{\mu,n}[\mf{p}_{\rho}] =&  
\frac{(n+m)^{\downarrow |\rho|}  \, (n+m-|\rho|)^{\downarrow n} m!}
{(n+m)! \, m^{\downarrow |\rho|}}
\mf{p}_\rho(\mu).
\end{align*}
A straightforward computation gives the desired expression.
\end{proof}

Substituting $\rho=\emptyset$ in Theorem \ref{thm:mu_average_M},
we obtain
$$
\bE_{\mu,n} [1] = 1,
$$
which shows that 
 $\mathbb{P}_{\mu,n}$ 
 defined in \eqref{eq:def_Plan_mu} is a probability measure on $\mcal{SP}_{n+m}$
 and that $\mathbb{E}_{\mu,n}$ is the average with respect to $\mathbb{P}_{\mu,n}$.

\begin{corollary} \label{cor:Av_p_rho}
For $\rho \in \mcal{OP}_k$, 
$$
\bE_n [\mf{p}_\rho ] = \delta_{\rho, (1^k)} n^{\downarrow k}.
$$
\end{corollary}

\begin{proof}
Notice that 
$\bE_n[\mf{p}_{\rho}]= n^{\downarrow m_1(\rho)} \bE_n[ \mf{p}_{\tilde{\rho}}]$
by \eqref{eq:mu_average_M_1}.
Substituting $\mu=\emptyset$ in Theorem \ref{thm:mu_average_M}
implies that $\bE_n[\mf{p}_{\tilde{\rho}}]= \delta_{\tilde{\rho},\emptyset}$.
Therefore $\bE_n[\mf{p}_{\rho}]$ survives 
only if $\rho=(1^{k})$, and 
$\bE_n[\mf{p}_{(1^{k})}] =n^{\downarrow k}$.
\end{proof}

Now the polynomiality of
$\bE_{\mu,n}[f]$ becomes trivial.

\begin{proof}[Proof of Theorems \ref{thm:ShPlPoly} and \ref{thm:ShPlPoly_mu}]
Since the $\mf{p}_{\rho}$, $\rho \in \mcal{OP}$, form a linear basis of the algebra
$\Gamma$ of supersymmetric functions, 
we obtain Theorem \ref{thm:ShPlPoly_mu} from Theorem \ref{thm:mu_average_M}.
Theorem \ref{thm:ShPlPoly}
is a special case of Theorem \ref{thm:ShPlPoly_mu} 
with $\mu=\emptyset$.
\end{proof}

\subsection{Orthogonality for  $\mf{p}_\rho$}

We can also easily compute the $\mathbb{P}_n$-average of 
products $\mf{p}_{\rho} \mf{p}_\sigma$.

\begin{thm} \label{thm:ortho_pp}
Let $\rho$ and $\sigma$ be odd partitions 
such that $m_1(\rho)=m_1(\sigma)=0$. Then 
$$
\bE_n [\mf{p}_{\rho} \mf{p}_{\sigma}] = \delta_{\rho, \sigma}
 2^{|\rho|-\ell(\rho)} z_{\rho}
 n^{\downarrow |\rho|}.
$$
\end{thm}

\begin{proof}
We may suppose $n \ge \max \{|\rho|, |\sigma| \}$
by virtue of Proposition \ref{prop:mp} (v).
It follows from Proposition \ref{prop:mp} (iii) that 
\begin{align*}
\bE_n [\mf{p}_{\rho} \mf{p}_{\sigma}] =&
\sum_{\lambda \in \mcal{SP}_n} \frac{2^{n-\ell(\lambda)} (g^\lambda)^2}{n!} \cdot 
\frac{n^{\downarrow |\rho|}}{g^\lambda} X^\lambda_{\rho
\cup (1^{n-|\rho|})} \cdot
\frac{n^{\downarrow |\sigma|}}{g^\lambda} X^\lambda_{\sigma
\cup (1^{n-|\sigma|})} \\
=& \frac{ n^{\downarrow |\rho|} n^{\downarrow |\sigma|}}{n!}
 \sum_{\lambda \in \mcal{SP}_n} 2^{n-\ell(\lambda)} 
X^\lambda_{\rho \cup (1^{n-|\rho|})}
X^\lambda_{\sigma\cup (1^{n-|\sigma|})}.  
\end{align*}
Using Proposition \ref{prop:OrthoXX} and 
the fact that $m_1(\rho)=m_1(\sigma)=0$, it equals 
$$
= \frac{ n^{\downarrow |\rho|} n^{\downarrow |\sigma|}}{n!}
2^{n-(\ell(\rho) +n-|\rho|)} z_{\rho \cup  
 (1^{n-|\rho|})} \delta_{\rho, \sigma}
=  n^{\downarrow |\rho|} 2^{|\rho| -\ell(\rho)}
z_{\rho} \delta_{\rho, \sigma}.
$$
\end{proof}

Substituting $\sigma=\emptyset$ in Theorem \ref{thm:ortho_pp} recovers 
Corollary \ref{cor:Av_p_rho}.

\subsection{Bound for degrees} \label{Bound_degree}

The following proposition is a direct consequence of 
Corollary \ref{cor:Av_p_rho}.

\begin{prop} 
Let $f$ be a supersymmetric function. If we expand $f$
as a linear combination of $\mf{p}_{\rho}$:
$$
f=\sum_{\begin{subarray}{c} \rho \in \mcal{OP} \\
\mathrm{(finite \ sum)} \end{subarray}} a_{\rho}(f) \mf{p}_{\rho},
$$
then 
$$
\bE_{n} [f]= \sum_{r \ge 0} a_{(1^r)}(f) n^{\downarrow r}.
$$
In particular, 
if $a_{(1^r)}(f)$ vanish for all $r >k$, then
$\bE_n[f]$ is of degree at most $k$. 
\end{prop}

Inspired by \cite{IvanovKerov},
we define a degree filtration of the vector space $\Gamma$ by
$$
\deg_1 (\mf{p}_{\rho})= |\rho|+ m_1(\rho).
$$
More generally, for $f=\sum_{\rho} a_{\rho}(f) \mf{p}_{\rho} \in \Gamma$,
define
$$
\deg_1 (f)= \max_{\rho: a_{\rho}(f) \not=0} (|\rho|+m_1(\rho)).
$$

\begin{prop} \label{prop:degree1}
The degree of $\bE_n[f]$ as a polynomial in $n$ is at most 
$\frac{1}{2} \deg_1(f)$.
\end{prop}

\begin{proof}
It is sufficient to check this for $f=\mf{p}_\rho$.
By virtue of Corollary \ref{cor:Av_p_rho}, we have:
if $\rho \not= (1^k)$ then $\bE_n [\mf{p}_{\rho}]=0$;
if $\rho=(1^k)$ then $\bE_n[\mf{p}_{\rho}]=n^{\downarrow k}$
and $\deg_1 (\mf{p}_{\rho})=2k$.
\end{proof}

\subsection{Examples}

We show some explicit expressions of $\bE_{n} [p_{\rho}]$,
which are presented in Subsection \ref{subsec:SPM}.
In Example \ref{example:p_to_mp}, we give expansions of some $p_{\rho}$
in $\mf{p}_{\sigma}$.
By Corollary \ref{cor:Av_p_rho}, we obtain the following identities
immediately. 
\begin{align*}
\bE_n[p_3] =& \bE_n[\mf{p}_{(3)} +3 \mf{p}_{(1^2)} + \mf{p}_{(1)}]
= 3n^{\downarrow 2} + n, \\
\bE_n[p_5] =& \bE_n[\mf{p}_{(5)} + 10 \mf{p}_{(3,1)} + 
\frac{35}{3} \mf{p}_{(3)} + \frac{40}{3} \mf{p}_{(1^3)} 
+15 \mf{p}_{(1^2)} + \mf{p}_{(1)}]  \\
=& \frac{40}{3} n^{\downarrow 3} + 15 n^{\downarrow 2} +n.
\end{align*}
Moreover, Theorem \ref{thm:ortho_pp} gives
\begin{align*}
\bE_{n} [p_3^2] 
=& \bE_{n} [(\mf{p}_{(3)} +3 \mf{p}_{(1^2)} + \mf{p}_{(1)})^2] \\
=& \bE_n[\mf{p}_{(3)} \mf{p}_{(3)} + 6 \mf{p}_{(3)} \mf{p}_{(1^2)}
+2 \mf{p}_{(3)} \mf{p}_{(1)} + 9 \mf{p}_{(1^2)} \mf{p}_{(1^2)} +
6\mf{p}_{(1^2)} \mf{p}_{(1)} + \mf{p}_{(1)} \mf{p}_{(1)}] \\
=& 12 n^{\downarrow 3}+0+0
 + 9 n^{\downarrow 2}\cdot n^{\downarrow 2} + 6 n^{\downarrow 2} \cdot n+ n^2 \\
=& 9 n^{\downarrow 4} + 54 n^{\downarrow 3} + 31 n^{\downarrow 2} + n.
\end{align*}

\section{Content evaluations} \label{sec:Content_Evaluation}

\subsection{Supersymmetry}

In this section, we give a proof of Theorem \ref{thm:content_evaluation}.
Let $\lambda$ be a strict partition and 
recall the shifted Young diagram 
$$
S(\lambda)= \{(i,j) \in \bZ^2 \ | \ 1 \le i \le \ell(\lambda), \ i \le j \le \lambda_i+i-1\}.
$$
For each $\square=(i,j) \in S(\lambda)$, 
its content $c_{\square}$ is defined by $c_{\square}=j-i$.
We find
\begin{align*}
\{c_\square \in \bZ \ | \ \square \in S(\lambda)\}
=&
\{j-i  \ | \ 1 \le i \le \ell(\lambda), \ i \le j \le \lambda_i+i-1 \} \\
=&  \{j-1 \ | \ 1 \le i \le \ell(\lambda), \ 1 \le j \le \lambda_i \}
%=&
%\{j-i  \ | \ 1 \le i \le \ell(\lambda), \ 1+i \le j \le \lambda_i+i \} \\
%=&  \{j \ | \ 1 \le i \le \ell(\lambda), \ 1 \le j \le \lambda_i \}
\end{align*}
as multi-sets.

\begin{lemma} \label{lem:p_exp1}
For each $m=0, 1,2,\dots$, 
$$
p_{2m+1} (\lambda)= \sum_{\square \in S(\lambda)} 
\{(c_\square+1)^{2m+1} -c_\square^
{2m+1}\}
%p_{2m+1} (\lambda)= \sum_{\square \in S(\lambda)} \{c_\square^{2m+1} -(c_\square-1)^
%{2m+1}\}
$$
for any strict partition $\lambda$.
\end{lemma}

\begin{proof}
Consider the function 
$$
\Phi(u;\lambda)= \prod_{i=1}^{\ell(\lambda)} \frac{1+\lambda_i u }{1-\lambda_i u}.
$$
The Taylor expansion of $\log \Phi(u;\lambda)$ at $u=0$ is
\begin{align*}
& \log \Phi(u;\lambda)= \sum_{i=1}^{\ell(\lambda)} \{\log(1+\lambda_i u) -
\log(1-\lambda_i u)\} \\
=& \sum_{i=1}^{\ell(\lambda)} \sum_{r=1}^\infty \frac{u^r}{r} \{
1+ (-1)^{r-1}\} \lambda^r_i 
= 2  \sum_{m =0}^\infty \frac{u^{2m+1}}{2m+1} 
p_{2m+1}(\lambda).
\end{align*}
On the other hand, since
$$ 
\frac{1+\lambda_i u }{1 -\lambda_i u} 
= \prod_{j=1}^{\lambda_i} \frac{(1+ju)(1- (j-1)u)}{(1-ju)(1+(j-1)u)},
$$
we see that
$$
\Phi(u;\lambda)=\prod_{i=1}^{\ell(\lambda)} \prod_{j=1}^{\lambda_i}
 \frac{(1+ju)(1- (j-1)u)}{(1-ju)(1+(j-1)u)}
  = \prod_{\square \in S(\lambda)} 
  \frac{(1+(c_\square+1) u)(1-c_\square u)}
  {(1-(c_\square+1) u)(1+c_\square u)}. 
% = \prod_{\square \in S(\lambda)} \frac{(1+c_\square u)(1-(c_\square-1)u)}{(1-c_\square u)(1+(c_\square-1)u)}.
$$
In this expression,
the Taylor expansion of $\log \Phi(u;\lambda)$ at $u=0$
is 
$$
\log \Phi(u;\lambda)= 2 \sum_{m=0}^\infty 
\frac{u^{2m+1}}{2m+1} \sum_{\square \in S(\lambda)} \{(c_\square+1)^{2m+1} 
-c_\square^{2m+1}\}.
%\log \Phi(u;\lambda)= 2 \sum_{m=0}^\infty 
%\frac{u^{2m+1}}{2m+1} \sum_{\square \in S(\lambda)} \{c_\square^{2m+1} 
%-(c_\square-1)^{2m+1}\}.
$$
Comparing coefficients in two expressions of the Taylor expansion, we obtain the desired formula.
\end{proof}

\begin{lemma} \label{lem:XtoY}
Let $R(X)$ be a polynomial in a variable $X$. 
Put $Y=X(X+1)$.
Then there exists a polynomial $\tilde{R}$ in $Y$
such that $\tilde{R}(Y)= R(X)$
if and only if
$R$ satisfies the functional equation $R(X)=
R(-X-1)$.
Moreover,
if the top-degree term of $R$ is $a X^{2m}$ then 
the top-degree term of $\tilde{R}$ is $a Y^m$.
\end{lemma}

\begin{proof}
First we suppose that $R(X)$ can be expressed as 
$R(X)= \tilde{R}(Y)$.
Since $Y=X(X+1)$ is invariant under the change of variable $X \mapsto \bar{X}=-X-1$,
we obtain the functional equation $R(X)= R(-1-X)$.

Next suppose that $R(X)$ satisfies the functional equation $R(X)=R(-1-X)$.
In general, a polynomial function $y=r(x)$ is symmetric with respect to the $y$-axis
in the $xy$-plane
if and only if $r$ is of the form
$$
r(x)= \sum_{j=0}^{m} a_j x^{2j}
$$
with certain coefficients $a_j$.
We may suppose $a_m \not=0$.
Put $R(X)= r(X+\frac{1}{2} )$.
We can observe that the symmetry $r(-x)=r(x)$ is equivalent to 
the functional equation $R(-X-1)=R(X)$.
Moreover, 
$R(X)$ is of the form
$$
R(X)= \sum_{j=0}^m a_j (X+\tfrac{1}{2} )^{2j} = \sum_{j=0}^m a_j
\(X(X+1) + \tfrac{1}{4}\)^{j}
=\sum_{j=0}^m  a_j (Y+\tfrac{1}{4})^j =: \tilde{R}(Y).
$$
The top-degree term of $R(X)$ is $a_m X^{2m}$, whereas
that of $\tilde{R}(Y)$ is $a_m Y^m$.
%A polynomial function $y=r(x)$ is symmetric with respect to the $y$-axis
%in the $xy$-plane
%if and only if $r$ is of the form
%$$
%r(x)= \sum_{j=0}^{m} a_j x^{2j}
%$$
%with certain coefficients $a_j$.
%We may suppose $a_m \not=0$.
%Put $R(X)= r(X+\frac{1}{2} )$.
%We can observe that the symmetry $r(-x)=r(x)$ is equivalent to 
%the functional equation $R(-X-1)=R(X)$.
%Moreover, 
%$R(X)$ is of the form
%$$
%R(X)= \sum_{j=0}^m a_j (X+\tfrac{1}{2} )^{2j} = \sum_{j=0}^m a_j
%\(X(X+1) + \tfrac{1}{4}\)^{j}
%=\sum_{j=0}^m  a_j (Y+\tfrac{1}{4})^j =: \tilde{R}(Y).
%$$
%The top-degree term of $R(X)$ is $a_m X^{2m}$, whereas
%that of $\tilde{R}(Y)$ is $a_m Y^m$.
\end{proof}

Define
$$
\widehat{c}_\square = \frac{1}{2} c_{\square}
(c_{\square}+1)
$$
for each $\square \in S(\lambda)$.

\begin{prop} \label{prop:hat_p}
Let $k=1,2,3,\dots$.
The function $\widehat{p}_k$ on $\mcal{SP}$ defined by
$$
\widehat{p}_k (\lambda) = p_k \( \widehat{c}_{\square}: \square \in S(\lambda)\)
=2^{-k} \sum_{\square \in S(\lambda)} \{ c_{\square}(c_{\square}+1) \}^k
$$
is supersymmetric.
Moreover,
if we set $\widehat{p}_0(\lambda)=|\lambda|$, then
$\widehat{p}_0$ is also supersymmetric.
\end{prop} 

\begin{proof}
First, we see that $\widehat{p}_0(\lambda)=|\lambda|=
\lambda_1+\lambda_2+\cdots=p_1(\lambda)$,
and hence $\widehat{p}_0$ is supersymmetric.

Let $m$ be a nonnegative integer and  $\lambda$ a strict partition.
Lemma \ref{lem:p_exp1} says that 
$$
p_{2m+1} (\lambda)= \sum_{\square \in S(\lambda)} \{(c_\square+1)^{2m+1} -c_\square^
{2m+1}\}.
$$
Since the polynomial function $R(X):=(X+1)^{2m+1} -X^{2m+1} = (2m+1) X^{2m} +
\cdots$ clearly satisfies the 
functional equation $R(X)=R(-X-1)$, it can be expressed 
as a polynomial in $Y=X(X+1)$ of degree $m$ by Lemma \ref{lem:XtoY}.
Hence there exist universal coefficients $a_{mr}$ $(r=0,1,2,\dots,
m-1)$ such that 
\begin{align*}
p_{2m+1}(\lambda)=& \sum_{\square \in S(\lambda)} 
\left[(2m+1) \{c_\square (c_{\square}+1)\}^m
 + \sum_{r=0}^{m-1} a_{mr} 
 \{c_\square (c_{\square}+1)\}^r \right] \\
=& 2^{m} (2m+1) \widehat{p}_{m}(\lambda)+ \sum_{r=0}^{m-1} 2^ra_{mr} \widehat{p}_r(\lambda).
\end{align*}
This relation implies that for each $k=0,1,2,\dots$, 
\begin{equation} \label{eq:hatp_to_p}
\widehat{p}_k= \frac{1}{2^k(2k+1)} p_{2k+1} + \sum_{r=0}^{k-1} b_{kr} p_{2r+1} 
\end{equation}
with some rational coefficients $b_{kr}$.
Therefore $\widehat{p}_k$ belongs to $\Gamma$.
\end{proof}

\begin{proof}[Proof of Theorem \ref{thm:content_evaluation}]
Let $F$ be any symmetric function.
It is well known that $F$ can be uniquely expressed as 
a polynomial in variables $p_1,p_2,\dots$.
Hence the function $\widehat{F}$ on $\mcal{SP}$ defined by
$$
\widehat{F}(\lambda)= F(\widehat{c}_{\square} \ : \ 
\square \in S(\lambda)) 
$$ 
is a polynomial 
in $\widehat{p}_{1},\widehat{p}_{2},\dots$.
Theorem \ref{thm:content_evaluation} follows from Proposition \ref{prop:hat_p}.
\end{proof}

The family $(\widehat{p}_k)_{k=0,1,2,\dots}$ 
is an algebraic basis of $\Gamma$
by \eqref{eq:hatp_to_p}.
This shows that
Theorem~\ref{thm:ShPlPoly_mu} and Corollary~\ref{cor:HX_polynomiality}
are {\it equivalent}.
Furthermore, we have obtained the following proposition.

\begin{prop}
The algebra $\Gamma$ coincides with 
the algebra generated by the function $\lambda \mapsto |\lambda|$ and 
the functions $\widehat{F}$, where $F$ are (ordinary) symmetric functions.
\end{prop}

\subsection{Examples} \label{subsec:ex_cont}
 
We give some examples of $\bE_{\mu,n}[ \widehat{F}]$,
where $F$ is a symmetric function.
It is easy to see that
$$
%\widehat{p}_{0}= p_1, \qquad 
\widehat{p}_1= \frac{1}{6} p_3- \frac{1}{6} p_1, \qquad 
\widehat{p}_2=  \frac{1}{20} p_5 - \frac{1}{12} p_3 + \frac{1}{30} p_1.
%\widehat{p}_1= \frac{1}{3} p_3- \frac{1}{3} p_1, \qquad 
%\widehat{p}_2=  \frac{1}{5} p_5 - \frac{1}{3} p_3 + \frac{2}{15} p_1.
$$
Indeed, for example, $\widehat{p}_1$ is computed as follows: 
$$
\widehat{p}_1(\lambda)= \sum_{\square \in S(\lambda)} \frac{1}{2}c_{\square}(c_\square+1)
=\frac{1}{2} \sum_{i=1}^{\ell(\lambda)} \sum_{j=1}^{\lambda_i} (j^2-j) 
= \sum_{i=1}^{\ell(\lambda)} \frac{\lambda_i^3-\lambda_i}{6} =
\frac{1}{6} (p_{3}(\lambda) - p_1(\lambda)).
$$
Using Example \ref{example:p_to_mp} and Corollary \ref{cor:Av_p_rho}, we have
$$
\bE_n [\widehat{p}_1] = 
\bE_n[\frac{1}{6} \mf{p}_{(3)} + 
\frac{1}{2} \mf{p}_{(1^2)}]=
\frac{1}{2}n^{\downarrow 2}.
$$
Similarly, we have
$$
\bE_n[\widehat{p}_2] = 
\bE_n[\frac{1}{20} \mf{p}_{(5)} + \frac{1}{2} \mf{p}_{(3,1)} +\frac{1}{2} \mf{p}_{(3)} +
\frac{2}{3} \mf{p}_{(1^3)} +\frac{1}{2} \mf{p}_{(1^2)}]= 
\frac{2}{3} n^{\downarrow 3} + \frac{1}{2} n^{\downarrow 2}.
$$
Moreover,  Theorem \ref{thm:ortho_pp} gives 
\begin{align*}
\bE_n[(\widehat{p}_1)^2] =& \bE_n[ \(\tfrac{1}{6} \mf{p}_{(3)} + \tfrac{1}{2} \mf{p}_{(1^2)}\)^2] \\
=& 
\frac{1}{36} \bE_n [ \mf{p}_{(3)} \mf{p}_{(3)}] + \frac{1}{6} \bE_n[ \mf{p}_{(3)} \mf{p}_{(1^2)}]+ \frac{1}{4}  \bE_n[\mf{p}_{(1^2)} \mf{p}_{(1^2)}] \\
=& \frac{1}{36}\cdot 12 n^{\downarrow 3} + \frac{1}{4} n^2 (n-1)^2 \\
=& 6 \binom{n}{4} + 8 \binom{n}{3} +\binom{n}{2}.
%\bE_n[(\widehat{p}_1)^2] =& \bE_n[\frac{1}{36} \mf{p}_{(3)} \mf{p}_{(3)} + 
%\frac{1}{6} \mf{p}_{(3)} \mf{p}_{(1^2)} + 
%\frac{1}{4}\mf{p}_{(1^2)} \mf{p}_{(1^2)}] \\
%=& \frac{1}{36} \cdot 12 n^{\downarrow 3} +
%n^2(n-1)^2 \\
%=& \frac{1}{12} (n^{\downarrow 4} + 4 n^{\downarrow 3}
%-8 n^{\downarrow 2}-2n).
\end{align*}

The following example is seen in \cite[Theorem 1.3]{HanXiong2015}:
$$
\bE_{\mu,n} [\widehat{p}_1 - \widehat{p}_1(\mu)] = \frac{1}{2} n(n-1) + n|\mu|.
$$
We can give its simple proof as follows.
Since $\widehat{p}_1 =\frac{1}{6} \mf{p}_{(3)} + \frac{1}{2} \mf{p}_{(1^2)}$,
we have $\widehat{p}_1(\mu)= \frac{1}{6} \mf{p}_{(3)}(\mu) + \frac{1}{2}|\mu|(|\mu|-1)$,
and $\bE_{\mu,n}[ \widehat{p}_1] = \frac{1}{6} \mf{p}_{(3)}(\mu)+ \frac{1}{2} 
(n+|\mu|)(n+|\mu|-1)$ by virtue of
Theorem \ref{thm:mu_average_M} and \eqref{eq:mu_average_M_1}.
Therefore we have
$$
\bE_{\mu,n} [\widehat{p}_1 - \widehat{p}_1(\mu)] = \frac{1}{2} (n+|\mu|)(n+|\mu|-1) -\frac{1}{2}
|\mu|(|\mu|-1) = \frac{1}{2} n(n-1) +n |\mu|.
$$

\section{Remarks on functions 
introduced by Han and Xiong
} \label{sec:Remark_HX}

%Let $\lambda$ be a strict partition and consider its shifted Young diagram
%$S(\lambda)$. 
We identify 
a strict partition
$\lambda$ with 
its shifted Young diagram
$S(\lambda)$ as usual.
A box $\square=(i,j)$ in $S(\lambda)$ is said to be an {\it outer corner}
of $\lambda$ if we obtain a new strict partition by removing the box $\square$ from $S(\lambda)$.
A box $\square \in \bZ^2$ is said to be an {\it inner corner}
of $\lambda$ if we obtain a new strict partition by adding the box $\square$ 
to $S(\lambda)$.
Denote by $\mathbb{O}_\lambda$ and by $\mathbb{I}_\lambda$ the
set of all outer and inner corners of $\lambda$, respectively.
For example,
if $\lambda=(5,4,2)$, then we have
$\mathbb{O}_\lambda=\{(2,5),(3,4)\}$
and $\mathbb{I}_\lambda=\{(1,6),(3,5), (4,4)\}$.

%We introduce the following notation.
%For a nonnegative integer $n$, we set 
%$$
%\widehat{n}:=\frac{1}{2} n(n+1).
%$$
For each integer $k \ge 1$,  
we define 
a function $\psi_k$ on $\mcal{SP}$ by 
$$
\psi_k(\lambda)= \sum_{\square \in \mathbb{I}_\lambda} \{ c_\square (c_\square+1)\}^k
-\sum_{\square \in \mathbb{O}_\lambda} 
\{ c_\square (c_\square+1)\}^k.
$$
In their paper \cite{HanXiong2015}, Han and Xiong first introduced
those functions.
Remark that 
these are denoted by
$q_k$ (or by $\Phi^k$)
 in their articles
with slight change $\psi_k=2^k q_k$.
Our purpose in this short section is 
to give an alternative simple expression 
of $\psi_k$
and to show that they are supersymmetric.

\begin{prop}
For each $\lambda \in \mcal{SP}$, we have
\begin{equation} \label{eq:ge_q}
\prod_{i =1}^{\ell(\lambda)} 
\frac{1-\lambda_i (\lambda_i-1) u}
{1- \lambda_i(\lambda_i+1) u}= \exp \( \sum_{k=1}^\infty \frac{u^k}{k} \psi_k(\lambda) \).
\end{equation}
\end{prop}

\begin{proof}
We prove the formula by induction on $|\lambda|$.
If $\lambda=\emptyset$, then \eqref{eq:ge_q} holds true because
%$\widetilde{\Phi}(u;\emptyset)=1$ and 
$\psi_k(\emptyset)=0$ for all $k \ge 1$.
Let $\lambda$ be a strict partition and put 
$$
\widetilde{\Phi}(u;\lambda)=
\prod_{i =1}^{\ell(\lambda)} 
\frac{1-\lambda_i (\lambda_i-1) u}
{1- \lambda_i(\lambda_i+1) u}
\qquad \text{and} \qquad
\widetilde{\mathbf{\Phi}} (u;\lambda)= 
\exp \( \sum_{k=1}^\infty \frac{u^k}{k} \psi_k(\lambda) \).
$$
Consider a strict partition $\lambda^+$
obtained by adding a box $\square$
to $\lambda$. 
The added box $\square$ is an inner corner of $\lambda$ and 
of the form $\square=(r, \lambda_r+r)$, where
$r$ is an integer in $\{1,2,\dots,\ell(\lambda)+1\}$ and we set
$\lambda_{\ell(\lambda)+1}=0$.
Notice that $c_{\square}= \lambda_r$.
It is easy to see that
$$
\frac{\widetilde{\Phi}(u;\lambda^+)}{\widetilde{\Phi}(u;\lambda)}=
\frac{(1-\lambda_r(\lambda_r+1) u)^2}{(1-(\lambda_r+1)(\lambda_r+2)u)
(1- (\lambda_r-1)\lambda_r u)}.
$$
If we show the same identity
for $\widetilde{\mathbf{\Phi}}$
then the induction step is completed.

A careful observation of the difference 
between $\mathbb{I}_\lambda \sqcup \mathbb{O}_\lambda$
and $\mathbb{I}_{\lambda^+} \sqcup
\mathbb{O}_{\lambda^+}$ implies that
\begin{align*}
\psi_k(\lambda^+)- \psi_k(\lambda)=& 
\{(c_{\square}+1)(c_{\square}+2)\}^k
+ \{(c_{\square}-1)c_{\square}\}^k -
2 \{c_{\square}(c_{\square}+1)\}^k \\
=&\{(\lambda_r+1)(\lambda_r+2)\}^k
+ \{(\lambda_r-1)\lambda_r\}^k -
2 \{\lambda_r(\lambda_r+1)\}^k 
\end{align*}
see also \cite[(3.10)]{HanXiong2015}.
Therefore we have
\begin{align*}
\frac{\widetilde{\mathbf{\Phi}} (u;\lambda^+)}
{\widetilde{\mathbf{\Phi}} (u;\lambda)} =&
\exp \( \sum_{k=1}^\infty \frac{u^k}{k} (\psi_k(\lambda^+)-\psi_k(\lambda))\)  \\
 =&  \frac{(1-\lambda_r(\lambda_r+1)
 u)^2}
 {(1-(\lambda_r+1)(\lambda_r+2)u)
(1- (\lambda_r-1) \lambda_r u)},
\end{align*}
as desired.
\end{proof}

The function $\psi_k$ is simply given as 
a polynomial in variables $\lambda_1,\lambda_2,\dots$.

\begin{prop} 
For each $k \ge 1$ and strict partition $\lambda$, we have
$$
\psi_k(\lambda)
=\sum_{i=1}^{\ell(\lambda)} \lambda_i^k\{(\lambda_i+1)^k-(\lambda_i-1)^k\}
=2 \sum_{\begin{subarray}{c}1 \le s \le k \\ \mathrm{odd}  
\end{subarray}} \binom{k}{s} p_{2k-s}(\lambda).
$$
In particular, $\psi_k$ is a supersymmetric function.
\end{prop}

\begin{proof}
Taking the logarithm of \eqref{eq:ge_q}, we find
$$
\sum_{k=1}^\infty \frac{u^k}{k} \psi_k(\lambda) = 
\sum_{i =1}^{\ell(\lambda)} \log  
\frac{1-\lambda_i (\lambda_i-1) u}{1- \lambda_i (\lambda_i+1)u}.
$$
Expanding the logarithm functions and 
comparing the coefficient of $u^k$ on both sides, we obtain the first equality in the theorem.
The remaining equality is obtained by
applying the binomial theorem for the first equality.
\end{proof}

For example:
$$
\psi_1= 2p_1, \qquad \psi_2= 4p_3, \qquad \psi_3=6 p_5+2  p_3, \qquad 
\psi_4= 8 p_7+ 8 p_5.
$$

\section{Open problems} \label{sec:OpenProblems}

\subsection{Degree functions on $\Gamma$}

In Subsection \ref{Bound_degree}, we introduced 
the filtration $\deg_1$ on the vector space $\Gamma$.
We remain a conjecture:

\begin{conj} 
The filtration $\deg_1$ is 
compatible with the multiplication of $\Gamma$ in the following sense.
Define the structure constants $f_{\sigma \tau}^\rho$ by 
$$
\mf{p}_{\sigma} \mf{p}_\tau = \sum_{\rho} f^\rho_{\sigma \tau} \mf{p}_\rho.
$$
For $f^\rho_{\sigma \tau} \not=0$ then,
$$
|\rho|+m_1(\rho) \le (|\sigma|+m_1(\sigma)) + (|\tau|+m_1(\tau)).
$$
Hence $\deg_1$ defines an algebra filtration of $\Gamma$.
\end{conj}

Remark that
the corresponding result in the algebra of shifted-symmetric functions
is obtained in \cite{IvanovKerov},
based on the theory of the partial permutation algebra.
Can we find a spin-analog of the partial permutation algebra?

We showed that, if $f$ is a supersymmetric function,
 $\mathbb{E}_n[f]$ is a polynomial in $n$ of degree 
at most $\frac{1}{2} \deg_1(f)$. 
To consider various degree filtrations is of help for 
estimations of the degree of $\bE_n[f]$, see \cite{DolegaFeray, IO2002}.

\subsection{Polynomiality for non-supersymmetric functions}

Assume that $F$ is a symmetric function but not supersymmetric.
It is natural to ask whether $\bE_n[F]$ is a polynomial in $n$.
As a trial, let us consider 
the second power-sum symmetric function $p_2(x_1,x_2,\dots)=\sum_{i \ge 1} x_i^2$.
Some values are directly computed as follows:
\begin{align*}
\bE_{1}[p_2]=&1, & \bE_2[p_2]=&4, & \bE_3[p_2]=& \frac{23}{3}, \\
\bE_4[p_2]=& 12, & \bE_5[p_2]=&17, & \bE_6[p_2]=&\frac{1016}{45}.
\end{align*}
Recall  $\bE_n[p_3]= 3n^2 -2n$
and a trivial identity $\bE_n[p_1]=n$.
If $\bE_n[p_2]$ is a polynomial in $n$, one may expect 
that it is of degree at most $2$.
However, there is no polynomial $\Phi_{p_2}(x)$ of degree 2 such that 
$\bE_{n}[p_2]=\Phi_{p_2}(n)$ for all $1 \le n \le 6$.

\subsection{Hook evaluations}

Recall the ordinary Plancherel measure $\mathbb{P}_n^{
\mathrm{Plan}}$ on partitions. 
As we mention in Subsection~\ref{subsec:related}, 
Stanley \cite{Stanley2010} (see also \cite{Han2009})  proves that
the summation 
$$
\sum_{\lambda \in \mcal{P}_n}\mathbb{P}_n^{\mathrm{Plan}} (\lambda) F(h_\square^2 : \square \in Y(\lambda))
$$
is a polynomial in $n$ for any symmetric function $F$, 
where $h_\square$ denotes the hook length of the square $\square$
in the Young diagram $Y(\lambda)$.
What is the analog of this result for the shifted Plancherel measure on strict partitions? 
%:Reference

%<<<<<<<<<<<<<<<<<<<<<<<<<<<<<<<<<<<<<<<<<

\bigskip

\noindent
\textsc{Sho Matsumoto} \\
Graduate School of Science and Engineering, Kagoshima University, \\ 
1-21-35, Korimoto, Kagoshima-city, Kagoshima, JAPAN. \\
\verb|shom@sci.kagoshima-u.ac.jp|

\end{document}